\documentclass{fundam}

\usepackage{graphicx,caption}
\usepackage{amscd,amssymb,latexsym} %dsfont}
\usepackage[utf8]{inputenc}
\newcommand\s[1]{\mathds{#1}}

\newcommand{\N}{{\mathbb N}}

\newcommand{\Z}{{\mathbb Z}}

\newcommand{\bz}{{\bf 0}}
\def\set#1{\{#1\}}
\def\pair#1{\langle{#1}\rangle}
\def\bz{{\bf 0}}
\def\ve{\varepsilon}
\def\s{\sigma}
\def\S{\Sigma}
\def\l{\lambda}

\def\restrict#1{\raise-.5ex\hbox{\ensuremath|}_{#1}}

\def\ca{{\mathcal A}}
\def\ct{{\mathcal T}}
\def\cp{{\mathcal P}}

\def\wid#1{|#1|}
\def\T{\Theta}

\definecolor{olive}{rgb}{0.5, 0.5, 0.0}

 \newcommand{\leaf}[1]{\begin{picture}(10,10)(5,10)\put(5,5){#1}\end{picture}}

 \newcommand{\lb}[1]{\begin{picture}(10,20)(10,20) \put(10,20){\circle*{2}}%
     \put(0,0){\line(1,2){10}}\put(0,0){#1}\end{picture}}
 \newcommand{\rb}[1]{\begin{picture}(10,20)(10,20)%
     \put(10,0){\line(-1,2){10}}\put(10,0){#1}\end{picture}}

\newcommand{\fut}[2]{\begin{picture}(20,20)%
    \put(0,0){\lb{#1}}\put(10,0){\rb{#2}}\end{picture}}

\newcommand{\Ttau}{\fut{\leaf {$c$}}{\leaf {$d$}}}

\newcommand{\Fut}[2]{\begin{picture}(40,10)(20,10)%
    \put(0,0){\line(2,1){20}}\put(0,0){#1}%
    \put(40,0){\line(-2,1){20}}\put(40,0){#2}%
 \end{picture}}

\begin{document}

\setcounter{page}{27}
\publyear{22}
\papernumber{2117}
\volume{186}
\issue{1-4}

   \finalVersionForARXIV
  %% \finalVersionForIOS

\title{Affine Completeness of Some Free Binary Algebras}

\author{Andr\'e Arnold\\
Universit\'e  de Bordeaux, France,\\
aa-labri@sfr.fr
\and Patrick C\'egielski\\
LACL\\
 Universit\'e  Paris-Est Cr\'eteil -- IUT de S\'enart-Fontainebleau, France \\
cegielski@u-pec.fr
\and  Ir\`ene Guessarian\thanks{Address of correspondence:  IRIF,  CNRS \& Universit\'e Paris-Diderot,
                                  Emerita Sorbonne Universit\'e, France}
 \\
IRIF,  CNRS \& Universit\'e Paris-Diderot,\\
Emerita Sorbonne Universit\'e, France\\
ig@irif.fr
}

\maketitle

\runninghead{A. Arnold et al.}{Affine Completeness}

\begin{flushright}
 \small{\it To the memory of Boris Trakhtenbrot, in honor of his  visionary\\
   contribution to theoretical computer science, logics and automata theory}
 \end{flushright}

\begin{abstract}
A function on an algebra  is congruence preserving if, for any
congruence, it maps pairs of congruent elements onto pairs of
congruent elements. An algebra
 is said to be affine complete if every congruence preserving function is a polynomial function. We show that  the algebra of  (possibly empty) binary trees
whose leaves are labeled by letters of an alphabet containing at
least one  letter, and the free monoid on an alphabet containing at least two letters are affine complete.
\end{abstract}

\runninghead{A. Arnold et al.}{Affine Completeness}

\section{Introduction}
%%%%%%%
%%%%%%%

A function on an algebra is {\em congruence preserving} if, for any
congruence, it maps pairs of congruent elements onto pairs of
congruent elements.

A {\em polynomial function} on an algebra is any function defined by
a term of the algebra using variables, constants and the operations
of the algebra. Obviously, every polynomial function is  congruence
preserving. An algebra is said to be {\em affine complete} if every
congruence preserving function is a polynomial function.

We proved in \cite{ALUN17} that if $\Sigma$ has at least three
elements,  then the free monoid $\Sigma^*$ generated by
$\Sigma$ is affine complete. If $\Sigma$ has just one letter $a$,
then the free monoid $a^*$ is isomorphic to
$\langle \N, + \rangle$, and we proved in~\cite{cggIJNT} that, e.g.,
$f \colon \N \to \N$ defined by
$f(x) = \texttt{if } x == 0 \texttt{ then }1 \texttt{ else }\lfloor e x!\rfloor$,
where $e = 2.718\dots$ is the Euler number, is congruence
preserving but not polynomial. Thus $\langle \N, + \rangle$, or equivalently the free
monoid $a^*$ with concatenation, is {\em not} affine complete.
Intuitively, this  stems from the fact that the more generators
$\Sigma^*$ has, the more congruences it has too: thus $\N$ with
just one generator, has very few congruences, hence many
functions, including non polynomial ones, can preserve all
congruences of $\N$. We also proved in \cite{acgg20} that, when
$\Sigma$ has at least three letters,  in the algebra of full binary trees with
leaves labelled by letters in $\S$, every unary congruence preserving function is polynomial (from now on, ``Congruence Preserving" is abbreviated as CP).
These previous works  left several open questions.
What happens if $\Sigma$ has one or two letters: for algebras of
trees? for non unary CP functions on trees? for the free monoid
 generated by   two letters? We answer those three questions in the present
paper: all these algebras are affine complete.

For full binary trees and at least three letters in $\S$, the proof of
\cite{acgg20} consisted in showing that CP functions which
coincide on $\S$ are equal, and in building for any CP function $f$
a polynomial $P_f$  such that $f(a)=P_f(a)$ for $a \in \S$,
wherefrom we inferred that $f = P_f$ for any $t$. We now generalize
this result in three ways: we consider arbitrary trees (with labelled
leaves) where the empty tree is allowed, the alphabet $\S$ may
have  one or two letters instead of  at least three, and CP functions of any
arity are allowed. Our method mostly uses congruences
$\sim_{u, v}$ which substitute for occurrences of a tree $u$ a
smaller tree $v$: in fact,   we even restrict ourselves to   congruences
such that $u$ belongs to a subset $\ct$ which is chosen in
a way ensuring that every congruence class has a unique smallest
canonical representative.
Using these congruences, we build, for each CP function $f$, and each $\tau \in \ct$, a
polynomial $P_\tau$ such that, for trees $u_1, \ldots, u_n$ small
enough, $f(u_1, \ldots, u_n) = P_\tau(u_1, \ldots, u_n)$. We furthermore
show that polynomials which coincide on $\S$ coincide on the
whole algebra, wherefrom we conclude that all the $P_\tau$ are
equal and $f$ is a polynomial.

\medskip
For the free monoid, it remains to answer the question: is $\{a, b\}^*$ equipped with concatenation affine
complete? We show in the present paper that the answer is
positive. The essential tool used in \cite{ALUN17} was the notion
of Restricted Congruence Preserving functions (RCP), i.e.,
functions preserving only the congruences defined by kernels
of endomorphisms
$\langle \Sigma^*, \cdot \rangle \to \langle \Sigma^*, \cdot \rangle$,
which allowed to prove that RCP functions are polynomial, implying that a
fortiori CP functions are polynomial. Unfortunately, the fundamental property
$\cp$ below,   which  was implicitly used when there are three letters,  no longer holds where there are only two letters.

\medskip
$(\cp)$\quad  {\parbox{0.85 \textwidth}{Let $\gamma_{a, b}$ be the homomorphism substituting
$b$ for $a$,
   if $f \colon \Sigma \to \Sigma $ is such that for all $a, b \,\in \S$,
        $\gamma_{a, b}(f(a)) = \gamma_{a,b}(f(b))$ then\\ $f$ is either
        a constant function, or the identity.}
}\medskip

Let $\S=\set{\s_1,\ldots,\s_n}$. When $n = 2$,  property ($\cp$)  no longer
holds hence restricting ourselves to RCP functions cannot help in
proving that CP functions are polynomial. For instance, the function
$f \colon \Sigma^* \to \Sigma^*$ defined by
$f(w) = \s_1^{|w|_{\s_1}}\cdots \s_n^{|w|_{\s_n}} $, where
$|w|_\sigma$ denotes the number of occurrences of the letter
$\sigma$ in $w$, is clearly neither polynomial, nor CP (the
congruence ``to have the same first letter'' is not preserved), even though it is RCP when $n=2$.
Fortunately $f$ is {\em } not RCP when $n\geq 3$, and thus it is not a
counter-example to the result stated in \cite{ALUN17}. Thus, for words in $\S^*$, we here  use a new
method, which also works even when $|\S| = 2$. It is very
similar to the method used for trees, even though the proofs are
more complex because of the associativity of the product
(usually called concatenation) of words.

Most of the proofs of
intermediate Lemmas and Propositions  are identical for trees and for words or have
only minor differences. Important differences, due to
the associativity or non associativity of the product in the
corresponding algebras, are located in the
proofs of just two Assumptions, that we prove separately.

The paper is thus organized as follows. In section 2, we recall the
basics about algebras, polynomials and congruence preserving
functions. In Section \ref{sec:length} we prove that the relation
between the  length of the value of a function and the length of its
arguments is affine for both CP functions and polynomials. In
Section \ref{sec:toolbox} we define the main kind of
congruences we will use and we show how to compute canonical
representatives for these congruences. In section
\ref{sec:cp-functions}, we define  polynomials
associated with a CP function and prove that CP functions are
polynomial assuming   two Assumptions. We prove these two
Assumptions for the algebra of trees (Section \ref{sec:tree}) and for  the free monoid $\{a.b\}^*$ (Section \ref{sec:word}).   Section~\ref{sec:word} ends with an application of
 the result  on lengths of Section  \ref{sec:length} which immediately implies the
  affine completeness of the free commutative monoid $\langle \N^p, +, 0\rangle$ for $p\geq 2$.

%%%%%%%
%%%%%%%
\section{ Binary algebras}
%%%%%%%
%%%%%%%
%%%%%%%
Let $\S$ be a nonempty finite alphabet, whose letters will be
denoted by $a, b, c, d, \ldots$.

We consider an algebraic structure $\pair{\ca(\S), \star, \bz}$, with
$\bz \notin \S$, subsuming both the free monoid and the set of
binary trees, satisfying the following axioms (Ax-1), (Ax-2), (Ax-3)
\begin{itemize}
\itemsep=0.9pt
\item[(Ax-1)] $\S \cup \set{\bz} \subseteq \ca(\S)$,
\item[(Ax-2)] if $u \notin \S \cup \set{\bz}$ then
 $\exists v, w \in \ca(\S): u = v\star w$.
\item[(Ax-3)] there exists a mapping $|\cdot| \colon \ca(\S) \to \N$
such that
\begin{itemize}\item $|\bz| = 0$,
\item $|\s| = 1$, for all $\s \in \S$,
\item $|u \star v| = |u| + |v|$.
\end{itemize}
\end{itemize}
\noindent $|u|$ is said to be the {\em length} of $u$ (it is equal to
the number of occurrences of letters of $\Sigma$ in $u$). We
similarly define, for $\sigma \in \Sigma$ and $u \in \ca(\S)$,
$|u|_\sigma$ which is the number  occurrences of the letter
$\sigma$ in $u$.

\medskip
The free monoid and the algebra of binary trees are examples of
such an algebra. If $\ca(\S)$ is the set of words $\S^*$ on the
alphabet $\Sigma$, $\star$ is the (associative) concatenation of
words, and $\bz$ is the empty word $\ve$, we get the free monoid.
If $\ca(\S)$ is the set of binary trees whose leaves are labelled by
letters of $\Sigma$, $t \star t'$ is a tree consisting of a root whose
left subtree is $t$ and  whose right subtree is $t'$, and $\bz$ is the
empty tree then we get the algebra of binary trees. In the case of
trees the operation $\star$ is not associative. The free commutative monoid
$\langle \N^p, +, (0,\ldots,0) \rangle$ is also a binary algebra satisfying
(Ax-1), (Ax-2), (Ax-3).

For our proofs the main difference between trees and the other
examples relates to point (Ax-2) above: the decomposition
$u = v \star w$ is unique for trees and not for the other examples.
\begin{fact} [Unicity of  decomposition]\label{uniquedecompo}
If $t$ is a tree not in $\set{\bz} \cup \, \Sigma$ then there exists a
unique ordered pair $\pair{t_1, t_2} \neq \pair{\bz, \bz}$ in $\+A^2$
such that $t = t_1 \star t_2$.
\end{fact}

An element of $\ca$ (a word or a tree) will be called an {\it object}.

%%%%%%%
\subsection{Polynomials}
\label{sec:polynomials}
%%%%%%%%

We denote by $\ca$ the set $\ca(\S)$. We also consider the infinite
set of variables $X = \set{x_i \mid i \geq 1}$, disjoint from $\S$. We
denote by $\ca_n$, the set $\ca(\S \cup \set{x_1, \ldots, x_n})$.
Note that $\ca = \ca_0$ and that $\ca_n \subseteq \ca_{n+1}$.

\begin{definition}
A $n$-ary {\em polynomial with variables} $\set{x_1, \ldots , x_n}$ is
an element $P$ of $\ca_n$. The {\em multidegree} of $P$ is the
$n$-tuple $\pair{k_1, \ldots, k_n}$ where $k_i  = |P|_{x_i}$.

\noindent With every such polynomial $P$ we associate a $n$-ary
{\em polynomial function} $\tilde P \colon \ca^n \to \ca$ defined
by:\\
for any $\vec u=\langle u_1,\ldots,u_i,\ldots,u_n\rangle \in \ca^n$,\medskip

$\tilde P(\vec u)=\left\{
    \begin{array}{ll}
      P  & \mbox{if $P = \bz$ or $P \in \S$}\\
      u_i& \mbox{if $P = x_i$}\\
      \widetilde {P_1}(\vec u) \star \widetilde {P_2}(\vec u)& \mbox{if $P = P_1 \star P_2$}
      \end{array}
  \right.$
\end{definition}

\paragraph{Note.} In the case of words we have to prove that the
value of $\widetilde{P}$ is independent of its decomposition
$P = P_1\star P_2$. This is due to the fact that
$\widetilde{P}(\vec{u})$ can be seen as a homomorphic image of
$P$ by an homomorphism from $\ca_n$ to $\ca$.

\medskip

From now on we simply write $P$ instead of $\tilde P$ for denoting
the function associated with the polynomial $P$.

%%%%%%%%%%%
\subsection{Sub-objects}
%%%%%%%%%%%

Let $\ca_{1,1}$ be the set of degree 1 unary polynomials with
variable $y$, i.e., elements $P\in  \ca(\S\cup\set{y})$ such that
$|P|_{y} = 1$,  or objects of $\ca(\S\cup\set{y})$ with exactly
one occurrence of $y$.
\begin{definition}\label{def:reducible}
An element $u$ of $ \ca$ is a {\em sub-object} of an element
$t \in \ca$, if there exists an occurrence of $u$ inside $t$, formally:
if there exists a polynomial $P \in \ca_{1,1}$ such that $P(u) = t$.
\end{definition}
In the case of words (resp. trees), sub-objects are  factors (resp.
subtrees).
\begin{definition}
A {\it sub-polynomial} $Q$ of a polynomial $P \in \ca_n$ is a
sub-object of $P$.
\end{definition}
%

%%%%%%%%%%%%%
%%%%%%%%%%%%%
\subsection{Congruence preserving functions}
%%%%%%%%%%%%%
%%%%%%%%%%%%%

\begin{definition}%\label{def:usimv}
A {\em congruence} on $\pair{\ca, \star,\bz}$ is an equivalence relation
$\sim$ compatible with $\star$, i.e., $s_1\sim s'_1$ and
$s_2\sim s'_2$ imply $s_1\star s_2\sim s'_1\star s'_2$.
\end{definition}
\begin{definition}\label{def1:cp}
A function $f \colon \ca^n \to \ca$ is {\em congruence preserving}
(abbreviated into CP) on $\langle \ca, \star,\bz\rangle$ if, for all
congruences $\sim$ on~$\langle \ca, \star,\bz\rangle$, for all
$t_1, \ldots, t_n, \ t'_1, \ldots, t'_n$ in $\ca$, $t_i \sim t'_i$ for all
$i = 1, \ldots, n$, implies $f(t_1, \ldots, t_n) \sim f(t'_1, \ldots, t'_n)$.
\end{definition}
Obviously,  every polynomial function is CP. Our goal is to prove the
converse, namely
\begin{theorem}\label{E}  Assume $|\S| \geq 2$ for words and
$|\S| \geq 1$ for trees.
If $f \colon \ca(\S)^n \to \ca(\S)$ is CP then there exists a
polynomial $P_f$ such that $f = \widetilde{P_f}$.
\end{theorem}
\noindent  This is the main result of the paper, which will be
proven in Sections \ref{sec:cp-functions}, \ref{sec:tree} and
\ref{sec:word}.

%%%%%%%%%%
%%%%%%%%%%
\section{Length condition}\label{sec:length}

For polynomials, as a consequence of (Ax-3), we get:
\begin{fact}\label{sec:length-conditions}
If $P\in \ca_n$ is a polynomial of multidegree
$\pair{k_1, \ldots, k_n}$ then

$|P(u_1, \ldots, u_n)| = |P(\bz, \ldots, \bz)|
 + \sum_{i = 1}^n k_i.|u_i|.$
\end{fact}

A necessary condition for a function $f \colon \ca^n \to \ca$ to be
polynomial is that $f$ has in someway a multidegree
$\pair{k_1, \ldots, k_n}$, playing the r\^ole of the multidegree of
polynomials, i.e.,  such that
$|f(u_1, \ldots, u_n)| = |f(\bz, \ldots, \bz)| + \sum_{i = 1}^n k_i.|u_i|.$
For words when $|\S| \geq 3$, the existence of such a multidegree
is proved in \cite{ALUN17}. We  here generalise this proof so that it
also applies to trees  and to smaller alphabets.

\begin{lemma}\label{l:lambdas}
  Let $f \colon \ca(\S)^n \to  \ca(\S)$ be a $n$-ary CP function.

  \medskip
(1) There exist functions $\l, \l_i \colon \N^n \to \N$ such that
$|f(u_1, \ldots, u_n)| = \l(|u_1|, \ldots, |u_n|)$ and
$|f(u_1, \ldots, u_n)|_i = \l_i(|u_1|_{i}, \ldots, |u_n|_{i})$, for
$i = 1, 2$. %%

 \medskip
(2) $\l(p_1 + q_1, \ldots, p_n + q_n) = \l_1(p_1, \ldots, p_n)
 + \l_2(q_1, \ldots, q_n)$.
\end{lemma}

\begin{proof}For an object $u \in \ca$, denote by $|u|_1 = |u|_{a}$
the number of occurrences of the letter $a$ in $u$, and let
$|u|_2 = |u| - |u|_1$. Formally, ${\wid \ve }_1 = 0$,
$\wid {a}_1 = 1$, $\wid {\s}_1 = 0$ for $\s\neq a$, and
${\wid {t\star t'}}_1 = {\wid {t}}_1 +  {\wid {t'}}_1 $.\medskip

(1) As the relation $|u| = |v|$ is a congruence and $f$ is CP,
$|u_i| = |v_i|$ for $i = 1, \ldots, n$ implies
$|f(u_1, \ldots, u_n)| = |f(v_1, \ldots, v_n)|$ hence
$|f(u_1, \ldots, u_n)|$ depends only on the lengths
$|u_1|, \ldots, |u_n|$, and $\l$ is well defined. Similarly for $\l_i$,
$i = 1, 2$ as $|u|_i = |v|_i$ is also a congruence.\medskip

(2) Consider objects $u_i$ with $|u_i|_1 = p_i$ and $|u_i|_2 = q_i$
(see Figure \ref{dessinArbu_1u_2}). On the one hand,
$|f(u_1, \ldots, u_n)| = \l(|u_1|, \ldots, |u_n|)
 = \l(p_1+q_1,\ldots,p_n+q_n)$,
$|f(u_1, \ldots, u_n)|_1 = \l_1(p_1, \ldots, p_n)$ and
$|f(u_1, \ldots, u_n)|_2 = \l_2(q_1, \ldots, q_n)$. On the other
hand, $|f(u_1, \ldots, u_n)| = |f(u_1, \ldots, u_n)|_1
 + |f(u_1, \ldots, u_n)|_2$, hence (2).
\end{proof}

\begin{figure}[!h]
\vspace*{-2mm}
\setlength{\unitlength}{1.5\unitlength}
\begin{center}
\begin{picture}(100,80)
  \put(0,60){\line(2,1){20}}
  \put(20,50){\line(2,1){20}}
  \put(40,40){\line(2,1){20}}
  \put(60,30){\line(2,1){20}}
  \put(60,10){\line(2,1){20}}
  \put(80,20){\line(2,1){20}}

  \put(20,70){\line(2,-1){100}}
  \put(80,20){\line(2,-1){20}}
  \put(60,10){\line(2,-1){20}}

   \put(-5,55){$a$}
   \put(15,45){$c$}
   \put(35,35){$a$}
   \put(55,25){$c$}

   \put(75,-5){$b$}
   \put(95,5){$a$}
   \put(115,15){$c$}
\end{picture}
\end{center}
\caption{A tree $u_i$ with $p_i = |u_i|_1 = 3$ and
 $q_i = |u_i|_2 = 4$. }
\label{dessinArbu_1u_2}\vspace*{-2mm}
\end{figure}
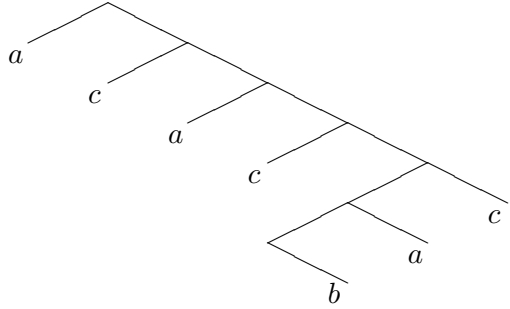

\begin{proposition}\label{ALUN}
For any $n$-ary CP function $f \colon \ca(\S)^n \to  \ca(\S)$, with
$|\S| \geq 2$, there exists a $n$-tuple $\pair{k_1, \ldots, k_n}$ of
natural numbers, called the {\em multidegree} of $f$, such that
$|f(u_1, \ldots, u_n)| = |f(\bz, \ldots, \bz)| + \sum_{i = 1}^n k_i.|u_i|$.
\end{proposition}
\begin{proof}
Let
$\vec{e_i} = \pair{\overbrace{0, \dots, 0}^{\text{$(i - 1)$ times}}, 1, 0, \ldots, 0}$,
$\vec{0}=\pair{0,\ldots,0}$, and apply Lemma \ref{l:lambdas}. We
have for any $m_1, \ldots, m_i,\allowbreak \ldots, m_n$, \vspace*{-2mm}
\begin{align*}
\l(m_1, \ldots, m_{i} + 1, \ldots, m_n)
 = \l_1(m_1, \ldots, m_i, \ldots, m_n) &+ \l_2(\vec{e_i}),\\
\l(m_1, \ldots, m_{i}, \ldots, m_n)
 = \l_1(m_1, \ldots, m_i, \ldots, m_n) &+ \l_2(\vec 0).
\end{align*}
Subtracting the second line from the first one:
$$\l(m_1, \ldots, m_{i} + 1, \ldots, m_n)
 - \l(m_1, \ldots, m_i, \ldots, m_n) = \l_2(\vec{e_i}) - \l_2(\vec 0).$$
\noindent Setting  $k_i = \l_2(\vec{e_i}) - \l_2(\vec 0)$,
 we get:
 \begin{align*}
\l(m_1, \ldots, m_{i}, \ldots, m_n)
 - \l(m_1, \ldots, m_i - 1, \ldots, m_n) &= k_i\\
&\ \ \vdots\\
\l(m_1, \ldots, 1, \ldots, m_n)
 - \l(m_1, \ldots, 0, \ldots, m_n) &= k_i\\
\text{Summing up the $m_i$ lines}\qquad\quad\qquad
\l(m_1, \ldots, m_{i}, \ldots, m_n)
 - \l(m_1, \ldots, 0, \ldots, m_n) &= k_i m_i\\
\text{\ Iterating for all  $i$ we get},\qquad\quad\qquad\qquad\qquad
\l(m_1, \ldots, m_n) - \l(\vec 0) = k_1m_1 + \cdots
 +& k_nm_n.
\end{align*}
Hence the result.
\end{proof}
Proposition \ref{ALUN} holds both for words and trees. However,
for trees the following better result holds even when $|\S| = 1$.

\begin{proposition}\label{ALUNone}
In the algebra of trees, for any $n$-ary CP function
$f \colon \ca(\S)^n \to  \ca(\S)$, there exists a $n$-tuple
$\pair{k_1, \ldots, k_n}$ of natural numbers, called the
{\em multi\-de\-gree} of $f$, such that
$|f(u_1, \ldots , u_n)| = |f(\bz, \ldots, \bz)| + \sum_{i = 1}^n k_i.|u_i|$.
\end{proposition}
\begin{proof}
For a tree $u \notin \S$,  $|u|_1$ (resp.  $|u|_2$) is the number
of  left (resp. right) leaves, so that $|u| = |u|_1 + |u|_2$ for
$u \notin \S$.
On Figure \ref{dessinArbu_1u_2} $|u_i|_1 = 4$ and
$|u_i|_2 = 3$.
Formally,  $|\bz| = |\bz|_1 = |\bz|_2 =0$.
For $u = t \star t' \notin \S$ we have\smallskip

$|u|_1 =  |t'|_1 + \left\{
  \begin{array}{ll}
     1& \mbox{if $t \in \S$},\\
     |t|_1& \mbox{if $t \notin \S$}.
  \end{array}
\right.$
and
$|u|_2 = |t|_2 + \left\{
  \begin{array}{ll}
  1& \mbox{if $t' \in \S$},\\
  |t'|_2& \mbox{if $t' \notin \S$}.
  \end{array}
\right.$

\medskip
We already know that the  relation $\sim$ defined by $u \sim v$
iff $|u| = |v|$ is a congruence.
For $j=1, 2$, the relation $\sim_j $ defined by
$u \sim_j v$ iff either $u = v \in \S$ or $u, v \notin \S$ and
$|u|_j = |v|_j$ is a congruence.
Hence if $f=\ca^n\to \ca$ is CP then for all
$u_1, \ldots, u_n, v_1, \ldots, v_n \notin \S$ such that
$\forall i = 1, \ldots, n,  |u_i|_j = |v_i|_j$ and
$f(u_1, \ldots, u_n), f(v_1, \ldots, v_n) \notin \S$, we have
$|f(u_1, \ldots, u_n)|_j = |f(v_1, \ldots, v_n)|_j$.
Without loss of generality, we may assume that for all
$u_1, \ldots,\allowbreak u_n$, $f(u_1, \ldots, u_n) $ is not in $\S$. This
holds because $g(u_1, \ldots, u_n) = \bz \star f(u_1, \ldots, u_n)$
is CP and   $|g(u_1, \ldots, u_n)| = |f(u_1, \ldots, u_n)|$.

For $u \notin \S$, we have $|u| = |u|_{1} + |u|_2$.
Exactly as in Proposition \ref{ALUN} we show that  for any
$m_1, \ldots, m_i, \ldots,\allowbreak m_n$,
$\l(m_1, \ldots, m_n) - \l(\vec 0) = k_1m_1 + \cdots + k_nm_n.$
It follows that for all $u_1, \ldots, u_n \notin \S$,
$|f(u_1, \ldots, u_n)| = |f(\bz, \ldots, \bz)| + \sum_{i = 1}^n k_i.|u_i|$.

\medskip
Finally, as for all $u \in \ca$, $u \star \bz \notin \S$ and
$|u \star \bz| = |u|$, we have:
$|f(u_1, \ldots, u_n)| = |f(u_1 \star \bz, \ldots, u_n \star \bz)|\allowbreak =
|f(\bz, \ldots, \bz)| + \sum_{i = 1}^n k_i.|u_i \star \bz|
 = |f(\bz, \ldots, \bz)| + \sum_{i = 1}^n  k_i.|u_i|$.
\end{proof}
\section{The toolbox}\label{sec:toolbox}
%%%%%%%%%%
%%%%%%%%%%

\subsection{Congruent substitutions}
\label{sec:congr-subst}

If $f$ is CP then $f(u) \sim f(v)$ as soon as $u \sim v$. This is why
we introduce specific congruences $\sim_{u, v}$ such that
$u\sim_{u, v} v$, so that if for some polynomial $Q$, (which is also
CP), we know that for some $u$, $f(u) = Q(u)$, then we know that
for all $v$, $f(v) \sim_{u, v} Q(v)$. Thus it  is important to describe
the congruence classes of such congruences.
\begin{definition}\label{def:usimv}
For ${u, v}$ a couple of objects in $\ca$ the relation $\sim_{u, v}$
is the equivalence relation generated by the set of pairs
$\set{\pair{P(u), P(v)} \mid P \in \ca_{1, 1}}$, i.e., $\sim_{u, v}$ is the least reflexive, symmetric and transitive relation containing all pairs ${\pair{P(u), P(v)}}$ with $P \in \ca_{1, 1}$.
$\sim_{u, v}$ is clearly a congruence on $\pair{\ca, \star,\bz}$.
\end{definition}
Given  such a congruence, we can consider the quotient algebra. It
may happen that each congruence class has a  simple canonical
representative.
For instance, the canonical representative could be the shortest
object in the congruence class, provided it is unique. However
unicity of the shortest representative certainly does not hold for the
congruences $\sim_{u, v}$ when $|u| = |v|$. It also happens that
unicity does not hold even when $|u| > |v|$ (Remark
\ref{rem:nonUnique}).
\begin{remark}\label{rem:nonUnique}
Even if $|u| > |v|$, there might be several  shortest congruent
elements. For instance in the case of words,
$ab\sim_{aa,b}aaa \sim_{aa,b} ba$, hence $ab$ and $ba$ are two
shortest elements congruent to $aaa$.
\end{remark}

\begin{definition}\label{def:reducible2}
For a given element $\tau$ of $\ca$, an element $t \in \ca$ is
$\tau$-{\em reducible}, if $\tau$ is a {\em sub-object} of $t$.
We denote by $\T_\tau$ the set of all $\tau$-irreducible objects in
$\ca$.
\end{definition}
In Figure \ref{fig:polirr}, $Q_\tau$ is $\tau$-reducible, $Q$ and
$P_\tau$  are $\tau$-irreducible, and in Figure \ref{fig:fak}, $t''$ is
$\tau$-irreducible.

\begin{figure}[!h]
\centering
\newcommand{\Tb}{\fut{\leaf {$b$}}{\leaf {$x_1$}}}
\newcommand{\Td}{\fut{\leaf {$x$}}{\leaf {$d$}}}
\newcommand{\TQ}{\fut{\lb{\leaf {$c$}}}{\Td}}
\newcommand{\TQp}{\Fut{\Tb}{\Ttau}}
\newcommand{\TP}{\fut{\Tb}{\leaf{$x_2$}}}

\begin{picture}(225,50)
\put(0,34){$\tau\!=$}\put(20,40){\Ttau}
\put(47,34){$Q\!=$} \put(70,40){\TQ}
\put(108,37){$Q_\tau\!=$}\put(140,40){\TQp}\put(140,40){\circle*{2}}
\put(182,35){$P_\tau\!=$}  \put(210,40){\TP}
\end{picture}

\caption{From left to right: tree $\tau = c \star d$, a
$\tau$-irreducible polynomial $Q$ with variable $x$, a
$\tau$-reducible polynomial $Q_\tau$ with variable $x_1$ together
with its associated $\tau$-irreducible polynomial
$P_\tau = Red^*_{\tau, x_2}(Q_\tau)$.}
\label{fig:polirr}\vspace*{-1mm}
\end{figure}
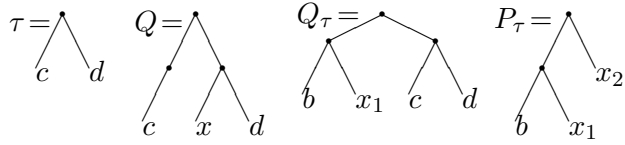

We now extend Definition \ref{def:reducible2} of $\tau$-irreducible
objects in $\ca$ to polynomials in $\ca_n$.
\begin{definition}\label{def:reducibleP}
Let $\tau \in \ca$. A polynomial  $P\in \ca_n$ is said to be
{\em $\tau$-irreducible} if any sub-object $v$ of $P$ which is in
$\ca$ is $\tau$-irreducible.
\end{definition}
Intuitively, the constant sub-objects (``coefficients'') of $P$ are
$\tau$-irreducible. In Figure \ref{fig:polirr}, $Q_\tau$ is the only
$\tau$-reducible polynomial.
\subsection{Canonical representatives}
\label{sec:canon-repr}
In fact it is possible to define and to ``compute'' a canonical
representative $t'$ of $t$ for $\sim_{\tau, v}$ if  $|\tau| > |v|$. To
this end we stepwise replace every occurrence of $\tau$ inside $t$
by $v$. To make this process deterministic we define the
{\it reduct} $Red_{\tau, v}(t)$ obtained by replacing by $v$ the
``leftmost'' occurrence of $\tau$ inside a $\tau$-reducible object
$t$.

\begin{definition}\label{def:RedTree}(Definition of $Red_{\tau, v}(t)$.) \smallskip \\
{\bf Case of trees} If $t = \tau$ then $ Red_{\tau, v}(t) = v$.
Oherwise, since $t \neq \tau$ is $\tau$-reducible,
$|t| > |\tau| \geq 1$, hence, by (Ax-2), $t = t_1 \star t_2$, and at
least one $t_i$ is $\tau$-reducible. Either $t_1 \in \ca$ is
$\tau$-reducible, and then
$Red_{\tau, v}(t) = Red_{\tau, v}(t_1) \star t_2$, or $t_1$ is
$\tau$-irreducible, then $t_2$ is $\tau$-reducible and
$Red_{\tau, v}(t) = t_1 \star Red_{\tau, v} (t_2)$.
Figure \ref{fig:fak} illustrates this reduction process.\smallskip

\noindent {\bf Case of words} Since $\tau$ is a factor of $t$, there
exists a shortest prefix $t'$ of $t$ such that $t = t'\tau t''$. Then
$Red_{\tau, v}(t) = t'vt''$.
\end{definition}
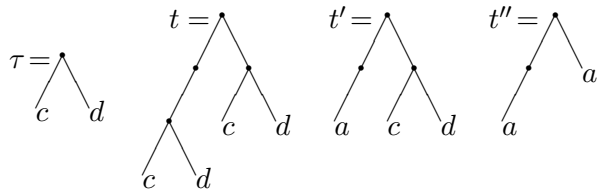
\begin{figure}[!h]\vspace*{-6mm}
\centering
\def\tdeux{\fut{\lb{\Ttau}}{\Ttau}}
\def\ttrois{\fut{\lb{\leaf{$a$}}}{\Ttau}}
\def\tquatre{\fut{\lb{\leaf{$a$}}}{\leaf{$a$}}}

\begin{picture}(210,70)
\put(0,40){$\tau\!=$}  \put(20,45){\Ttau}
\put(60,55){$t=$}     \put(80,60){\tdeux}
\put(120,55){$t'\!=$}  \put(142,60){\ttrois}
\put(180,55){$t''\!=$} \put(205,60){\tquatre}
\end{picture}
\caption{From left to right, $\tau = c \star d$,
$t = ((c \star d) \star \bz) \star (c \star d)$,
$t' = (a \star \bz) \star (c \star d) = Red_{\tau, a}(t)$,
$t'' = Red_{\tau, a}(t') = (a \star \bz) \star a$.}
\label{fig:fak}
\end{figure}
We iterate this partial reduction function to get a mapping
$Red^*_{\tau, v} \colon \ca \to \T_\tau$ inductively defined by:
$$Red_{\tau,v}^*(t)=\left\{
  \begin{array}{ll}
    t & \mbox{if $t\in \T_\tau$}\\
    Red_{\tau,v}^*(Red_{\tau,v}(t))&\mbox{if $t\notin \T_\tau$}.
  \end{array}\right.$$
\begin{proposition}\label{prop:red*G}
$Red^*_{\tau, v}(u \star w)
 = Red^*_{\tau, v} (Red^*_{\tau, v}(u) \star w).$
\end{proposition}
\begin{proof}
By definition, $Red_{\tau, v}^*(t) = Red_{\tau, v}^k(t)$, where $k$ is
the least integer such that $Red_{\tau, v}^k(t)$ is $\tau$-irreducible.
If $Red_{\tau, v}^*(u \star w) = Red_{\tau, v}^p(u \star w)$ and
$Red_{\tau, v}^*(u) = Red_{u, v}^q(u)$, necessarily $q\leq p$
and we have by induction on $i = 0, \ldots, q$,
$Red^{p}_{\tau, v}(u \star w) =
 Red^{p-i}_{\tau, v}(Red^i_{\tau, v}(u) \star w)$
hence the result for $i = q$.
\end{proof}
Although $Red^*_{\tau, v}(t)$ is a canonical representative of the
congruence class of $t$ modulo $\sim_{\tau, v}$, it is not
necessarily the only object of the equivalence class of $t$ having
minimal length, as shown in Remark~ref{rem:nonUnique}.

To prevent such situations, we will first  define for each algebra
a suitably chosen subset $\ct$ of the algebra  ensuring that for
each $\tau \in \ct$, there exists a unique canonical
representative of shortest length in the class of $\sim _{\tau, v}$ for
each $v \in \ca$ such that $|v| < |\tau|$ (Proposition
\ref{prop:sim}).
This set $\ct$ has to satisfy the following assumption.

\begin{assumption} \label{f:exUniCanRep}
$\forall  \tau \in \ct,\  v \in \ca,\  P \in \ca_{1, 1}$,\ \
$Red_{\tau, v}^*(P(\tau)) = Red_{\tau, v}^*(P(v)).$
\end{assumption}

 Proposition   \ref{ass1.t} (resp. \ref{ass1.w}) shows that this assumption holds for the set $\ct$ of trees defined by
\eqref{tau:tree} in Section \ref{sec:tree} (resp.  the set $\ct$ of words defined by
\eqref{tau:word} in Section \ref{sec:word}).

Provided the truth of this assumption, we get:

\begin{proposition}\label{prop:sim}(Existence of a canonical
representative)
Let $\tau\in\ct$, and $v \in \ca$ with $|\tau| > |v|$. For any
$t, t' \in \ca$,
$t \sim _{\tau, v} t' \ \text{ iff } \ Red_{\tau, v}^*(t)
 = Red_{\tau, v}^*(t').$
\end{proposition}
\begin{proof}
By the definition of $Red^*_{\tau, v}$, for all $t,t'$,
$t \sim_{\tau, v}Red^*_{\tau, v}(t)$, and
$t' \sim_{\tau, v} Red_{\tau, v}^*(t')$. Hence
$Red_{\tau, v}^*(t) = Red_{\tau, v}^*(t')$ implies $t\sim _{\tau, v} t'$
by transitivity.

\medskip
Conversely, if $t \sim _{\tau, v}t'$ then there exist
$t_1 = t,\; t_2, \ldots, t_n = t'$, and $P_i \in \ca_{1, 1}$ (see
Definition \ref{def:usimv}) such that for each  $i= 1,\ldots, n-1$,
$t_i = P_i(\tau)$ and $t_{i + 1} = P_i(v)$ (or vice-versa).
By Assumption \ref{f:exUniCanRep},
$Red_{\tau, v}^*(t_i) = Red_{\tau, v}^*(t_{i+1}) $,
hence $Red_{\tau, v}^*(t) = Red_{\tau, v}^*(t')$.
\end{proof}
\begin{proposition}\label{prop:reduc}
Let $\tau \in \ct$, $t$ and $t'$ be two objects such that $|v| < |\tau|$,
$t \sim_ {\tau, v} t'$, and $|t| < |\tau|$. Then $t = t'$ if and only if
$|t| = |t'|$.
\end{proposition}
\begin{proof}
If $t = t'$ then obviously $|t| = |t'|$.
Since $t \sim_ {\tau, v} t'$, by Proposition \ref{prop:sim},
$Red_{\tau, v}^*(t) = Red_{\tau, v}^*(t')$. But $|t'| = |t| < |\tau|$
implies that both $t'$ and $t$ are $\tau$-irreducible, hence
$t = Red_{\tau, v}^*(t) = Red_{\tau, v}^*(t') = t'$.
\end{proof}
%%%
\subsection{Strong irreducibility}
\label{sec:strong-irred}
%%%
By Propositions \ref{prop:sim} and \ref{prop:reduc}, we get that if
$|t| < |\tau|$ and $|Red^*_{\tau, v}(t')| > |\tau|$ then
$t \not\sim_{\tau, u} t'$. To prove that if $|t'| > |\tau|$ then
$|Red^*_{\tau, v}(t')| > |\tau|$, it is enough to prove that if $t'$
contains a sub-object $w$ of length $n \geq |\tau|$ then $w$ is a
sub-object of $Red^*_{\tau, v}(t')$.
This leads to the following definition.

\begin{definition}\label{def:str}
Let $\tau \in \ca$, an object $w$ is said to be {\em strongly
$\tau$-irreducible} if  $|w|\geq|\tau|$ and if whenever $w$ is a
sub-object of some $t \in \ca$, $w$ also is a sub-object of
$Red^*_{\tau, v}(t)$ for any $v$ such that $|v| < |\tau|$.
\end{definition}
We finally state the following assumption on $\ct$, the truth of
which is proven in Proposition \ref{ass2.t} (resp. \ref{ass2.w}) for
trees (resp. for words).
\begin{assumption}\label{prop:exs}
For all $\tau \in \ct$ and for all $\tau$-irreducible unary polynomials
$P$ of degree $k$ such that $|\tau| \geq 2k + 4$, we have the
following property:

If for all $u\in \ca$ such that $\wid u  \leq 1$, $P(u)$ is
$\tau$-reducible, then there exists $\theta \in \ca$ of length 1 and a
strongly $\tau$-irreducible sub-object $w$ of $P(\theta)$ of length
not less than $|\tau|$ (i.e., $|w| \geq |\tau|$).
\end{assumption}

%%%%%%%%%
%%%%%%%%%
\section{Proof of the main theorem}
\label{sec:cp-functions}
%%%%%%%%%
%%%%%%%%%

From now on, we postulate the existence of a set $\ct$ which
satisfies Assumptions \ref{f:exUniCanRep} and \ref{prop:exs}.

\subsection{The induction hypothesis}

The polynomiality of CP functions will be  proved by induction on
their arity. The basic step of this induction is obvious and common
to all algebras we consider: a function of arity $0$ is a constant,
which is a polynomial function.

For the inductive step, note that if $n \geq 0$ and $f$ is a
$(n + 1)$-ary CP function of multidegree
$\pair{k_1, \ldots, k_n, k_{n+1}}$, then for all $t$, $f_t$ defined by
$f_t(u_1, \ldots, u_n) = f(u_1, \ldots, u_n,t)$ is CP with multidegree
$\pair{k_1, \ldots, k_n}$, hence the induction hypothesis:

\begin{fact}\label{hypInduction}
\fbox{\begin{minipage}{0.8\textwidth}
{\bf Induction hypothesis.} For any $t \in \ca$, there exists a
polynomial $Q_t$ of multidegree $\pair{k_1, \ldots, k_n}$ such that:

\qquad $\forall u_1, \ldots, u_n \in \ca, \quad Q_t(u_1, \ldots, u_n)
 = f(u_1, \ldots, u_n, t).$
\end{minipage} }
\end{fact}

\begin{definition}
The {\it polynomial $P_\tau$ associated with $f$ and $\tau \in \ct$}
is the unique $\tau$-irreducible polynomial of multidegree
$\pair{k_1, \ldots, k_n, m}$ such that
$$\forall u_1, \ldots, u_n \in \ca,\  P_\tau(u_1, \ldots, u_n, \tau)
 = Q_\tau(u_1, \ldots, u_n) = f(u_1, \ldots, u_n, \tau).
$$
It is also defined by $P_\tau = Red^*_{\tau, x_{n+1}}(Q_\tau)$,
considering $P_\tau$ and $Q_\tau$ as objects in
$\ca( \S \cup \set{x_1, \ldots, x_n, \allowbreak x_{n+1}})$.
\end{definition}
Figure \ref{fig:polirr} illustrates this definition in the algebra of binary
trees.
%%%
%%%
\subsection{Partial polynomiality of CP functions}
%%%
%%%
Assuming the hypothesis stated in Fact \ref{hypInduction}, we can
proceed and prove
\begin{proposition}\label{prop:Ptau}
Let $\tau \in \ct$. If $\wid u < \wid \tau$ and if
$\wid{f(u_1, \ldots, u_n, u)} < \wid \tau$ then
\begin{itemize}
\itemsep=0.9pt
\item $f(u_1 ,\ldots, u_n, u)
 = Red^*_{\tau, u}(P_\tau(u_1, \ldots, u_n, u))$

\item either $m = k_{n+1}$ and
$f(u_1, \ldots, u_n, u) = P_\tau(u_1, \ldots, u_n, u)$, \\or
$m < k_{n+1}$ and $P_\tau(u_1, \ldots, u_n, u)$ is $\tau$-reducible.
\end{itemize}
\end{proposition}
\begin{proof}
Obviously, $f(u_1, \ldots, u_n, u) \sim_{\tau, u}
 f(u_1, \ldots, u_n, \tau) = P_\tau(u_1, \ldots, u_n, \tau)$
 $\sim_{\tau,u}P_\tau(u_1,\ldots,u_n,u)$.
As
$\wid{f(u_1, \ldots, u_n, u)} < \wid \tau$, $f(u_1, \ldots, u_n, u)$
is $\tau$-irredu\-cible. %\\
Thus, by Assumption \ref{f:exUniCanRep}, \\
$f(u_1, \ldots,\allowbreak  u_n, u)
 = Red^*_{\tau, u}(P_\tau(u_1, \allowbreak \ldots, \allowbreak u_n, u)).$
Let
$d = |f(u_1, \ldots, u_n, \tau)|
 = |P_{\tau}(u_1, \ldots, u_n, \tau)|$.
Then
$|f(u_1, \ldots, u_n, \allowbreak u)| = d - k_{n+1}(|\tau| - |u|)$ and
$|P_{\tau}(u_1, \ldots, u_n, u)| = d - m(|\tau| - |u|)$.

By Proposition \ref{prop:reduc},
$P_\tau(u_1, \ldots, u_n, u) = f(u_1, \ldots, u_n, u)$ if and only if
$|P_\tau(u_1, \ldots, u_n, u)| = |f(u_1, \ldots, u_n, u)|$ if and only if
$m = k_{n + 1}$.

Since
$f(u_1, \ldots, u_n, u) = Red^*_{\tau, u}(P_\tau(u_1, \ldots, u_n, u))$,
if $f(u_1, \ldots, u_n, u) \neq P_\tau(u_1, \ldots, u_n, u)$ then
$P_\tau(u_1, \ldots, u_n, u)$ is not $\tau$-irreducible.\\
Hence
$d - m(|\tau| - |u|) = |P_\tau(u_1, \ldots, u_n, u)| \geq |\tau|>
    |f(u_1, \ldots, u_n, u)| = d -  k_{n+1} (|\tau| - |u|)$, which implies
$m < k_{n + 1}$.
\end{proof}
An immediate consequence of Proposition \ref{prop:Ptau} is:
\begin{proposition}\label{p:appl}
Let $\tau \in \ct$, let $\pair{k_1, \ldots, k_n, m}$ be the multidegree
of $P_\tau$. Then
\begin{enumerate}
\itemsep=0.9pt
\item either $m = k_{n+1}$ and  for all $u \in \ca$ such that
$\wid u \leq \wid \tau$, and for all $u_1, \ldots, u_n \in \ca$ such that
$|f(u_1, \ldots, u_n, u)| < |\tau|$, we have\\
$P_\tau(u_1, \ldots, u_n, u) = f(u_1, \ldots, u_n, u)$,

\item or $m < k_{n+1}$ and for all $u \in \ca$ such that
$\wid u \leq \wid \tau $, and for all $u_1, \ldots, u_n \in \ca$ such
that $|f(u_1, \ldots, u_n, u)| < |\tau|$, $P_\tau(u_1, \ldots, u_n, u)$
is $\tau$-reducible.
\end{enumerate}
\end{proposition}
%
%%%
%%%
\subsection{Polynomiality of CP functions}
%%%
%%%
We first prove that for almost all $\tau$ we are in case (1) of
Proposition \ref{p:appl}.
\begin{proposition}\label{prop:m=k}
Let $\pair{k_1, \ldots, k_n, k_{n+1}}$ be the multidegree of $f$, let
$k = k_1 + \cdots + k_n + k_{n+1}$, and let $\tau \in \ct$  be  such that
$|\tau | \geq 2k + 4$. For all $u \in \ca$ such that $\wid u < \wid \tau$
and for all $u_1, \ldots, u_n \in \ca$ such that
$|f(u_1, \ldots, u_n, u)| < |\tau|$, we have
$P_\tau(u_1, \ldots, u_n, u) = f(u_1, \ldots, u_n, u)$.
\end{proposition}
\begin{proof}
By Proposition \ref{p:appl} it is enough to prove that $m <k_{n+1}$
is impossible.\smallskip

Let $P_\tau$ be the $\tau$-irreducible polynomial associated with
$\tau$ of multidegree $\pair{k_1, \ldots, k_n, m}$ and let us assume
that  $m <  k_{n+1}$. Then, by Proposition \ref{p:appl}, we have:
for all $u \in \ca$ such that $\wid u \leq \wid \tau$ and
$|f(u, \ldots, u, u)| < |\tau|$, the object $P_\tau(u, \ldots, u, u)$ is
$\tau$-reducible.

We now consider the $\tau$-irreducible unary polynomial $P'_\tau$
of degree $M = k_1 + \cdots + k_n + m < k$, obtained by
substituting $x_1$ for any variable $x_i$  in $P_\tau$.
Since $P'_\tau(u)$ is $\tau$-reducible for all $u$ such that
$|u| \leq 1 < |\tau|$, by Assumption \ref {prop:exs} there exist
$\theta$ of length $1$ and a strongly $\tau$-irreducible sub-object
$w$ of $P'_\tau(\theta) = P_\tau(\theta, \ldots, \theta, \theta)$ of
length not less than $\tau$. By Proposition
\ref{prop:Ptau}, $w$ is a sub-object of
$Red^*_{\tau, \theta}(P_\tau(\theta, \ldots, \theta, \theta))
 = f(\theta, \ldots, \theta, \theta)$.
Hence $|w| \!\leq \! |f(\theta, \ldots, \theta, \theta)|\!< \!|\tau|\!\leq \! |w|$, contradiction.
\end{proof}
Let $\tau_1$ and $\tau_2$  be such that $|\tau_i|> |f(a,\ldots,a)|$. Then,
by Proposition \ref{prop:m=k}, we have :\\
For all $ u_1, u_2, \dots, u_{n}, u$ such that $|u|$ and
$|f(u_1, \ldots, u_n)|$ are less that $|\tau_1|$ and $|\tau_2|$ then
\begin{eqnarray}\label{eq:1}
{P}_{\tau_1}(u_1, \ldots, u_{n}, u) = &f(u_1, \ldots, u_n, u)
 &={P}_{\tau_2}(u_1, \ldots, u_{n},u).
\end{eqnarray}
We first prove that $P_{\tau_1} = P_{\tau_2}$ as a consequence
of the next Proposition by observing  that equation (\ref{eq:1}) holds
for all $u_i$, $u$ of length 1.

\begin{proposition}\label{prop:P=Q}
Let $P$, $Q$ be polynomials of multidegree
$\pair{k_1, \ldots, k_n}$. \\
If, for all $u_1, u_2, \dots, u_{n}$ of length 1,  \
${P}(u_1, \ldots, u_{n}) = {Q}(u_1, \ldots, u_{n})$ then $P = Q$.
\end{proposition}
\begin{proof}
For a polynomial $P$ {\em in the algebra of trees}, we define
$s(P)$ to be the number of symbols of
$\S \cup \{\star\} \cup \set{x_1, \ldots, x_n}$ occurring in $P$.
Formally $s(\bz) = 0$, $s(a) = 1$ for
$a \in \S \cup \set{x_1, \ldots, x_{n}}$, and
$s(u \star v) = 1+ s(u) + s(v)$.
For $P$ {\em  in the algebra of words}, we set $s(P) = |P|$.

\medskip
In both cases there exists at least two distinct objects of length 1:
either two distinct letters $a,\; b$, or the trees $a \star \bz$ and
$\bz \star a$.

The proof is by induction on $s(P)$.

\medskip {\bf Basis.}

\noindent (1)
If $s(P) = s(Q) = 0$ then $P = \bz = Q$.\medskip

\noindent (2) If $s(P) = s(Q)  = 1$ then
$P, \; Q \in \S \; \cup \; \set{x_1, \ldots, x_n}$. If $P$ and $Q$ are
both constants, the result follows from  equality
$P(u, \ldots, u) = Q(u, \ldots, u)$. If $P = x_i$ and $Q = x_j$ with
$i \neq j$, the hypothesis
$P(u_1,  \ldots, u_n) = Q(u_1,  \ldots, u_n)$ leads to a
contradiction, as soon as $u_i\neq u_j$, hence $i = j$.
If $P$ is a constant $u$ and $Q$ is a variable $x_i$, we have
$u = P(u', \ldots, u') = Q(u', \ldots, u') =  u '$, a contradiction
when $u \neq u'$.

\medskip {\bf Inductive step.} If $s(P) > 1$ then  $P = P_1 \star P_2$ and
$Q = Q_1 \star Q_2$, (taking $|P_1| = |Q_1| =1$ in case of words).
For any $ u_1, u_2, \dots, u_{n}$ of length 1, we have
${Q}(u_1, \ldots, u_{n}) = {P}(u_1, \ldots, u_{n})
 = {P_1}(u_1, \ldots, u_{n}) \star {P_2}(u_1, \ldots, u_{n})
 = {Q_1}(u_1, \ldots, u_{n}) \star {Q_2}(u_1, \ldots, u_{n})$
which implies $P_i(u_1, \ldots, u_{n}) = Q_i(u_1, \ldots, u_{n})$,
hence, by the induction hypothesis, $P_1 = Q_1$ and
$P_2 = Q_2$, and thus $P = Q$.
\end{proof}
\begin{theorem}\label{prop:avd}
Let $f$ be a CP function of multidegree
$\pair{k_1, \ldots, k_n, k_{n+1}}$. There exists a polynomial $P_f$
of multidegree $\pair{k_1, \ldots, k_n, k_{n+1}}$ such for all
$u_1, \ldots, u_n$, $u \in \ca$,
$P_f(u_1, \ldots, u_n, u) = f(u_1, \ldots, u_n, u)$.
\end{theorem}
\begin{proof}
By Propositions \ref{prop:m=k} and \ref{prop:P=Q} there exists a
unique polynomial $P_f$ such that for all $\tau$ of length greater
than $|f(a, a, \ldots, a)|$, we have $P_\tau = P_f$.
For any $u_1, \ldots, u_n, u$, there exists $\tau$ such that
$|\tau|> \max(|u|, |f(u_1, \ldots, u_n,u )|)$. \smallskip \\
By Proposition \ref{prop:m=k},
$f(u_1, \ldots, u_n, u)
 = P_\tau(u_1, \ldots, u_n, u) = P_f(u_1, \ldots, u_n, u)$.
\end{proof}

%%%%%%%%%%%
%%%%%%%%%%%
\section{The case of trees}
\label{sec:tree}
%%%%%%%%%%%
%%%%%%%%%%%

We here consider the algebra of binary trees with labelled leaves.
For this algebra of trees we set
\begin{eqnarray} \label{tau:tree}
 \ct &=& \{\;\tau\in\ca \mid |\tau| \geq 2\;\}
\end{eqnarray}

\begin{proposition}\label{p:ir->sirr}
If a tree $w$ is $\tau$-irreducible, then it is strongly
$\tau$-irreducible.
\end{proposition}
\begin{proof}
By definition of $Red^*_{\tau, v}$, it is enough to show that if $w$ is
a subtreee of $t$ then it is a subtree of $Red_{\tau,v}(t)$.
The proof is by induction on $|t|$ such that $w$ is a subtree of $t$.
If $t$ is $\tau$-irreducible then $Red_{\tau, v}(t) = t$ and the result
is proved.
Otherwise, $t = t_1 \star t_2$, with $w$ subtree of some $t_i$, and
$Red_{\tau, v}(t) = Red_{\tau, v}(t_1) \star t_2$ or
$Red_{\tau, v}(t) =  t_1 \star Red_{\tau, v}(t_2)$. In both cases,
$w$ is a subtree of $Red_{\tau, v}(t)$.
\end{proof}

%%%%%%%%%%%%
\subsection{Canonical representative}\label{sec:canT}
%%%%%%%%%%%%

For trees, we can improve Proposition \ref{prop:red*G}.

\begin{proposition}\label{prop:red*}
$Red^*_{\tau, v}(u \star w)
 = Red^*_{\tau, v}(Red^*_{\tau, v}(u) \star Red^*_{\tau, v}(w)).$
\end{proposition}
\begin{proof}
By taking Proposition \ref {prop:red*G} into account, we just have to
prove that
$Red^*_{\tau, v}(u \star w)
 = Red^*_{\tau, v}(u \star Red^*_{\tau, v}(w))$
when $u$ is $\tau$-irreducible. This a consequence of the definition
of the leftmost reduction for trees:
$Red_{\tau, v}(u \star w) = u \star Red_{\tau, v}(w)$.
\end{proof}
We now prove that Assumption \ref{f:exUniCanRep} holds for our
algebra of binary trees.
\begin{proposition}\label{ass1.t}
$\forall P \in \ca_{1,1} \qquad Red_{\tau, v}^*(P(\tau))
 = Red_{\tau, v}^*(P(v)).$
\end{proposition}
\begin{proof}
The proof is by induction on $|P|$. If $P = y$ then
$Red_{\tau, v}^*(\tau) = Red_{\tau, v}^*(v) = v$.

If $P = P_1 \star P_2$ then by Proposition \ref{prop:red*},
\begin{align*}
Red_{\tau, v}^*(P(\tau)) &=
Red_{\tau, v}^*( Red_{\tau, v}^*( P_1(\tau)) \star
 Red_{\tau, v}^*(P_2(\tau))), {\hbox {and }}\\
Red_{\tau, v}^*(P(v))
 &=Red_{\tau, v}^*( Red_{\tau, v}^*( P_1(v)) \star
  Red_{\tau, v}^*(P_2(v))).
\end{align*}
Then, by the induction hypothesis,
$Red_{\tau, v}^*( P_i(v)) =  Red_{\tau, v}^*(P_i(\tau))$,  for $i=1,2$,  and thus \\
$Red_{\tau, v}^*( P(v))$ =  $Red_{\tau, v}^*(P(\tau))$.
\end{proof}
\subsection{Strongly irreducible trees}
\label{sec:polT}
%%%%%%%%%%%%

The following Proposition assures that Assumption
\ref{f:exUniCanRep} holds for trees.

\eject

\begin{proposition}\label{ass2.t}
For all $\tau \in \ct$  and for all
$\tau$-irreducible unary polynomials $P$ the following property
holds.

If for all $u \in \ca$ such that $\wid u \leq 1$, $P(u)$ is
$\tau$-reducible, then there exists $\theta \in \ca$ of length 1 and a
strongly $\tau$-irreducible subtree $w$ of $P(\theta)$ of length not
less than $|\tau|$ (i.e., $|w| \geq |\tau|$).
\end{proposition}
\begin{proof}
Let $\tau \in \ct$, which has length at least 2. Let $P$ be a non
constant $\tau$-irreducible polynomial such that for all $u \in \ca$
with length $\wid u \leq 1$, $P(u)$ is $\tau$-reducible. Let
$\s \in \S$, and let $t=\s\star \bz$ and $t'=\bz\star \s$, $t \neq t'$.

As $P(t)$ is $\tau$-reducible, it must contain $\tau$. But since $P$
is $\tau$-irreducible, there exists a non constant sub-polynomial
$Q$ of $P$ such that $Q(t) = \tau$. Then $|Q(t)| = |Q(t')| = |\tau|$
and, as $Q$ is non-constant, $Q(t') \neq \tau$. It follows that $Q(t')$
is $\tau$-irreducible, hence  strongly $\tau$-irreducible by
Proposition \ref{p:ir->sirr}. We set $\theta = t'$ and $w = Q(t')$.
\end{proof}
%
% %%%%%%%%%%%
%%%%%%%%%%%
\section{The case of words}
\label{sec:word}
%%%%%%%%%%%
%%%%%%%%%%%
For words, proving Assumptions \ref{f:exUniCanRep} and
\ref{prop:exs} requires more work because unicity of the
decomposition fails in the free
monoid. \smallskip \\
As shown in Remark \ref{rem:nonUnique}, Assumption
\ref{f:exUniCanRep} does not hold for any word $\tau$.
Indeed, Assumption \ref{f:exUniCanRep} fails as soon as $\tau$
self-overlaps, i.e., when there exists a word $t$ which is both a
strict prefix and a strict suffix of $\tau$. For instance, if
$\tau = aba$, $ab \sim_{aba, \ve} ababa \sim_{aba, \ve} ba$,
while $Red_{aba, \ve}(ab) = ab \neq ba = Red_{aba, \ve}(ba)$.
Obviously, words such that $a^nb^n$ do not self-overlap and thus
satisfy Assumption \ref{f:exUniCanRep}.
But we also need that these words satisfy Assumption
\ref{prop:exs}. The condition that $\tau$ is not self-overlapping is
not sufficient to satisfy Assumption \ref{prop:exs}. For instance, let
$\tau = aabb$ and $P = aa x_1 bb$,  which is $\tau$-irreducible.
The factors of length $\geq 4$ of $P(a) = aaabb$ and
$P(b) = aabbb$ are $aaabb,\  aabbb,\ aabb,\ aaab,\ abbb$. None
of them is strongly $\tau$-irreducible: $aaabb,\  aabbb,\ aabb$ are
$\tau$-reducible, and $aaab,\ abbb$ satisfy one of the forbidden
property (1) or (2) of Proposition \ref{prop:str}.
We thus have to introduce a stronger constraint to define a
suitable $\ct$, which turns out to be
\begin{eqnarray}
\ct &=& \{a^nbab^n \mid n >1\}  \label{tau:word}
\end{eqnarray}

%%%%%%%%%%%
\subsection{Canonical representative}
\label{sec:canW}
%%%%%%%%%%%
\begin{proposition}\label{ass1.w}
For all $P$ in  $\ca_{1,1}$ and  $\tau=a^nbab^n\in\ct $,\qquad
 $Red_{\tau, v}^*(P(\tau)) = Red_{\tau, v}^*(P(v))$.
\end{proposition}
\begin{proof}
The proof is by induction on $|P|$. \medskip

{\em Basis}. If $P = y$ then
$Red_{\tau, v}^*(\tau) = Red_{\tau, v}^*(v) = v$. \medskip

{\em Induction}. Let $P = uyw$ and let
$s = Red_{\tau, v}^*(u) \in \T_\tau$. By Proposition \ref{prop:red*G},
$Red_{\tau, v}^*(P(\tau)) = Red_{\tau, v}^*(s\tau w)$ and
$Red_{\tau, v}^*(P(v)) = Red_{\tau, v}^*(svw)$. Thus, to prove the
result it is enough to show that $Red_{\tau, v}(s\tau w) = svw$, i.e.,
that the shortest prefix $s\tau$ of $s\tau w$ is $s\tau$. Let us
assume that there exists $s'$  such that $s'\tau$ is a strict prefix
of $s\tau$. Since since $s \in \T_\tau$, $s'\tau$ is not a prefix of
$s$.

\begin{figure}[!h]
\centering
\begin{picture}(250,50)(0,0)
   \put(0,){\line(1,0){250}}
   \put(0,15){\line(1,0){250}}
   \put(0,30){\line(1,0){250}}
   \put(0,45){\line(1,0){250}}
   \put(70,30){\line(0,1){15}}
   \put(150,15){\line(0,1){30}}
   \put(120,0){\line(0,1){30}}
   \put(200,15){\line(0,-1){15}}
   \put(0,0){\line(0,1){15}}
   \put(0,30){\line(0,1){15}}

   \put(60,7){$s$}
   \put(160,7){$\tau$}
   \put(35,37){$s'$}
   \put(110,37){$\tau$}

   \put(135,22){$t$}
\end{picture}
\end{figure}

\noindent It follows that there exists a nonempty word $t$, with
$0 < |t| <|\tau|$, which is both a suffix and a prefix of
$\tau = a^nbab^n$, such that $s'\tau = st$.

The first letter of $t$ has to be $a$ and its last letter $b$.
Therefore $a^nb$ is a prefix of $t$ and $ab^n$ is a suffix of $t$,
hence $t = a^nbab^n$, contradicting  $ |t| <|\tau|$.
\end{proof}
%
% %
%%%%%%%%%%%%
\subsection{Strongly irreducible words}
%%%%%%%%%%%%

We state a sufficient condition for a word $w \in \ca$  to be strongly
$\tau$-irreducible.
\begin{proposition}\label{prop:str}
A nonempty word $w$ is strongly $\tau$-irreducible if it is
$\tau$-irreducible and it has the additional properties that $\tau$
and $w$ do not overlap, i.e., there do not exist words $u,  t', t$ such
that $t \notin \set{\ve, \tau}$  and
\begin{enumerate}
\itemsep=0.9pt
\item either  $w = ut$ and $\tau = tt'$,
\item or $\tau = t't$ and $w = tu$.
\end{enumerate}
\end{proposition}
\begin{proof}
It is enough to show that if a factor $w$ of $t$ satisfies the above
hypothesis, then $w$ is a factor of $Red_{\tau, v}(t)$ when
$|v| < |\tau|$.

Let $t = w'\tau w''$ with $w'$ $\tau$-irreducible. Then
$Red_{\tau, v}(t) = w'vw''$.
As $w$ is $\tau$-irreducible and $w$ and $\tau$ do not overlap, if
$w$ is a factor of $t$, it is a factor of $w'$ or a factor of $w''$,
hence a factor of $Red_{\tau, v}(t) = w'vw''$.
\end{proof}
The following proposition implies Assumption \ref{prop:exs}.
\begin{proposition}\label{ass2.w}
For all $\tau = a^nbab^n \in \ct$ and for all $\tau$-irreducible unary
polynomials $P$ of degree $k$ such that $|\tau| \geq 2k + 4$, the
following property holds.

If $P(\ve)$ is $\tau$-reducible, then there exists
$\theta \in \set{a, b}$ and a strongly $\tau$-irreducible sub-object
$w$ of $P(\theta)$ of length greater than $|\tau|$ (i.e.,
$|w| > |\tau|$).
\end{proposition}
\begin{proof}
Let $\tau = a^nbab^n \in \ct$ and let $P$ be a $\tau$- irreducible
polynomial of degree $k$  such that $P(\ve), P(a)$, and $P(b)$ are
$\tau$-reducible. Note that since $|\tau| = 2n + 2$ the condition
$|\tau| \geq 2k + 4$ is equivalent to $n - 1 > k$.

\medskip
Since $\tau$ is a factor of $P(\ve)$ there exists a factor $Q$ of $P$
such that $Q(\ve) = \tau$, i.e.,
$$Q = ax^{p_1}ax^{p_2}a\cdots ax^{p_n}
b x^m a  x^{q_1}bx^{q_2}b\cdots x^{q_n}b$$
with $k=p + m + q < n - 1$, where $p = p_1 + p_2 + \cdots + p_n$
and $q = q_1 + q_2 + \cdots + q_n$.

\eject
We show that at least one of the words $Q(a)$ or $Q(b)$ is
strongly $\tau$-irreducible.

We first show that if
$Q(a) = a^{n+p}ba^{1+m+q_1}ba^{q_2}b\cdots a^{q_n}b$ is not
strongly $\tau$-irreducible, then $m = q = 0$.

\medskip
If $Q(a)$ is not strongly $\tau$-irreducible, then it is either
$\tau$-reducible and we are in case (i) below, or it is
$\tau$-irreducible and then we are in one of cases (ii) or (iii) below.
\begin{itemize}
\itemsep=0.85pt
\item[(i)] $Q(a)$ is $\tau$-reducible, i.e., $\exists u, v$ such that:
 $Q(a) = u\tau v$, or
\item [(ii)] $Q(a) = u t$ and $\tau = t v$, with  $v \neq \ve \neq t$
(Proposition \ref{prop:str} (1)),  or
\item[(iii)] $Q(a) = t v $ and $\tau = u t$, with $u \neq \ve \neq t$
(Proposition \ref{prop:str} (2)).
\end{itemize}
For both Cases (ii) and (iii), as both $Q(a)$ and $\tau$ start with
$a$ and end with $b$,  the first letter of $t$ is $a$ and its last letter
is $b$.

\medskip
{\em Case(i)} If $\tau$ is a factor of $Q(a)$ then $bab^n$ is a factor
of $Q(a)$. The only factor of $Q(a)$ starting and ending with $b$,
ending with $b$, and containing $(n+1)$ $b$'s is
$ba^{1+m'+m_1}ba^{m_2}b\cdots a^{m_n}b$, which implies
$m' + m_b = 0$.

\medskip
{\em Case(ii)} Assume now $\exists u, v, t$ with $Q(a) = u t$ and
$\tau = t v$, with $v \neq \ve$. As $t$ is a prefix of $\tau$, we have
$t = a^nb$ or $t = a^nbab^{n'}$ with $0 < n' < n$. Since $t$ is a
suffix of $Q(a)$, in all cases, $a^nb$ is a factor of $Q(a)$. As
for all $i$ $q_i \leq q < n - 1$ and, since
$1 + m + q_1 \leq 1 + p+ m + q < 1 + (n - 1) = n$, the unique suffix
of $Q(a)$ starting with $a^nb$ is
$t = a^nba^{1+m +q_1}ba^{q_2}b\cdots a^{q_n}b$. Since $t$ is a
prefix of $\tau$, we have
$n + 1 + m + q = |t|_a \leq |\tau|_a = n + 1$, which implies
$m=q= 0$.

\medskip
{\em Case(iii)} Assume now $\exists u, v, t$ with $Q(a) =t v $ and
$\tau = u t$, with $u \neq \ve$. Since $t$ is a suffix of $\tau$,
then either $t = ab^n$ or $t = a^{n'}bab^n$ with $0 < n' < n$.
Since $t$ is a prefix of $Q(a)$, $a^{n + p}b$ is also a prefix of
$t$. Both cases are impossible since $n + p > n' \geq 1$.

Hence if $Q(a)$ is not strongly $\tau$-irreducible, $m = q = 0$.

\medskip
By a symmetrical reasoning on
$Q(b) = ab^{p_1}ab^{p_2}\cdots ab^{p_n+m+q}a b^{q_n+n}$
we get that if $Q(b)$ is not strongly $\tau$-irreducible, then
$p = m =0$.

Finally, if both $Q(a)$ and $Q(b)$ are not strongly
$\tau$-irreducible then  $p = m = q = 0$, hence $\tau$ is a factor of
$P$, contradicting the $\tau$-irreducibility of $P$. Thus, either
$Q(a)$ or $Q(b)$ is strongly $\tau$-irreducible. Then choose
$\theta \in \{a, b\}$ such that $w = Q(\theta)$ is strongly
$\tau$-irreducible.
\end{proof}
Hence, Theorem \ref{E} holds and if $|\S| \geq 2$ then $\S^*$ is
affine complete. Our proof method can  be extended to the free
commutative monoid with $p$ generators when $p \geq 2$ as
shown in the next subsection.

%%%%%%%%
\subsection{Application to free commutative monoids}
%%%%%%%%
Note that the free commutative monoid with $p$ generators is
isomorphic to $\N^p$. We now prove a variant of Proposition
\ref{ALUN} which immediately implies that the  commutative
binary algebra $\langle \N^p, +, \vec 0\rangle$ is affine complete,
thus giving a very simple proof of already known results
\cite{nobauer, werner}. \smallskip

For $u = \pair{\ell_1, \ldots, \ell_p} \in \N^p$ let
$|u| = \ell_1 + \cdots + \ell_p$ and $|  u|_j = \ell_j$ for
$i = 1, \ldots, p$.

\begin{proposition}\label{ALUNc}
For any  $n$-ary CP function  $f\colon (\N^p)^n \to  \N^p$, with
$p \geq 2$, there exists a  $n$-tuple $\pair{k_1, \ldots, k_n}$ of
natural numbers, called the {\em multidegree} of $f$, such that\medskip

(i) $|f(  u_1, \ldots,   u_n)| = |f(\bz, \ldots, \bz)|
 + \sum_{i = 1}^n k_i.|  u_i|$, and \smallskip

(ii) for all $j=1,\ldots,p$, $|f(  u_1, \ldots,   u_n)|_j
 = |f(\bz, \ldots, \bz)|_j + \sum_{i = 1}^n k_i.|   u_i|_j$
\end{proposition}
\begin{proof}
The proof is almost identical to the proof of Proposition \ref{ALUN}.
We stress here the differences. For an object
$u = \pair{\ell_1, \ldots, \ell_p}\in \N^p$, and an arbitrary element
$j \in \pair{1, \ldots, p}$, let us denote: $|u| = \ell_1+\cdot + \ell_p$,
$|u|_1 = \ell_j$, and $|u|_2 = |u| - |u|_1$.
There exist $\l, \l_1$ such that $\l(m_1, \ldots, m_n)$ is the
common value of all $|f(u_1, \ldots, u_n)|$ and
$\l_1(m_1, \ldots, m_n)$ is the common value of all
$|f(u_1, \ldots, u_n)|_1 = \ell_j$ for an arbitrary $j \in \{1, \dots, p\}$.
Lemma \ref{l:lambdas} and (i) are then proved as in Proposition
\ref{ALUN}.
Moreover
\begin{align*}
\l(m_1, \ldots, m_n)&= \l_1(m_1, \ldots, m_n) + \l_{2}(0, \ldots, 0)\\
&= \l_1(m_1, \ldots, m_n) - \l_1(0, \ldots, 0) + \l_1(0, \ldots, 0)
 +\l_{2}(0, \ldots, 0)\\
&= \l_1(m_1, \ldots, m_n) - \l_1(0, \ldots, 0)
 + \l(0, \ldots, 0) \  \text{[Lemma \ref{l:lambdas} 2]}
\end{align*}
Hence
$\l(m_1, \ldots, m_n) - \l(0, \ldots, 0) = \l_1(m_1, \ldots, m_n)
 - \l_1(0, \ldots, 0)$
which, as $\l_1$ can be any arbitrarily chosen $\l_j$, immediately
implies (ii).
\end{proof}
\begin{corollary}
The commutative algebra $\langle \N^p, +, 0\rangle$ is affine
complete.
\end{corollary}
\begin{proof}
Proposition \ref{ALUNc} (ii) means that the $j$th component
$|f(  x_1, \ldots,   x_n)|_j$ of $f(  x_1, \ldots,   x_n)$
is of the form $c_j+\sum_{i = 1}^n k_i.|   x_i|_j$, for all
$j = 1, \ldots, p$. Hence
$f(  x_1, \ldots,   x_n) =   c+ \sum_{i = 1}^n k_i.   x_i$
is indeed a polynomial.
\end{proof}
%%%%%%%
%%%%%%%
\section{Conclusion}
%%%%%%%
%%%%%%%

It  is  known that, when the alphabet has just one letter, the free
monoid is not affine complete \cite{cggIJNT}.
It is also known that, when the alphabet has at least two letters,
the free commutative monoid is affine complete  since it is isomorphic to
 a free module or a vector space of dimension at least 2, known to be
 affine complete \cite{nobauer, werner}.

 We here prove that
the (non commutative) free monoid $\S^*$ is affine complete as soon as its alphabet has at
least two letters (generalizing \cite{ALUN17} where the result was proved for $|\S|\geq 3$).

We also prove that the algebra of binary trees with labelled leaves
is affine complete for every nonempty finite alphabet $\S$, i.e., not assuming that $|\S|\geq 2$.
This difference with the case
of the free monoid might seem surprising. However since its product is not associative,
the algebra of trees has more structure, hence more congruences, and thus less CP functions, than the free monoid.

\end{document}